\documentclass[11pt,reqno]{amsart}
\usepackage[usenames,dvipsnames,svgnames,table]{xcolor}
\usepackage{amssymb}
\usepackage{amsfonts}
\usepackage{amsmath,amscd}
\usepackage{mathrsfs}
\usepackage{verbatim}
\usepackage{enumitem}

\usepackage{mathtools}
\usepackage[all]{xy}
\usepackage{hyperref}

\newtheorem{theorem}{Theorem}[section]
\newtheorem{lemma}[theorem]{Lemma}     
\newtheorem{corollary}[theorem]{Corollary}
\newtheorem{proposition}[theorem]{Proposition}
\newtheorem*{theorem*}{Theorem}
\newtheorem*{corollary*}{Corollary}

\theoremstyle{definition}
\newtheorem{remark}[theorem]{Remark}
\newtheorem{example}[theorem]{Example}
\newtheorem{notation}[theorem]{Notation}

\newcommand{\mbk}{\mbox{$\Bbbk$}}

\newcommand{\mbn}{\mbox{$\mathbb{N}$}}

\newcommand{\se}{\mbox{$\langle$}}
\newcommand{\sd}{\mbox{$\rangle$}}

\newcommand{\rank}{\mbox{${\rm rank}$}}
\newcommand{\card}{\mbox{${\rm card}$}}

\title[Binomial determinants]{Binomial determinants: some closed formulae}

\author[L. Gonz\'alez]{Laura Gonz\'alez}
\author[F. Planas-Vilanova]{Francesc Planas-Vilanova}
\address{Departament de Matem\`atiques, Universitat Polit\`ecnica de Catalunya. 
Diagonal 647, ETSEIB, E-08028 Barcelona.}
\email{laura.gonzalez.hernandez@upc.edu}
\email{francesc.planas@upc.edu}
\thanks{Partially supported by projects reference PID2019-103849GB-I00 and PID2023-146936NB-I00 funded by the Spanish State Agency MICIU/AEI/10.13039/501100011033 and by AGAUR grant 2021 SGR 00603}

\date{\today}


\begin{document}

\begin{abstract}
This paper is intended to give closed formulae for binomial determinants with consecutive or almost consecutive rows or columns, as well as calculating the generator of left nullspaces defined by some binomial matrices.  In the meantime, we reprove, by different means, the positivity of binomial determinants shown by Gessel and Viennot.
\end{abstract}

\maketitle 

\section{Introduction}\label{sec-intro}

The genesis of this paper is in the study of the work \cite{moh1}, where Moh finds a family of prime ideals 
$P_n$, $n$ odd,  in the power series ring $\mbk[[x,y,z]]$, $x,y,z$ variables over a field $\mbk$ of characteristic zero, such that each $P_n$ needs at least $n$ generators. Concretely, the ideal $P_n$ is the kernel of the map $\rho_n: \mbk[[x, y, z]]\rightarrow \mbk[[t]]$ defined by $\rho_n(x)=t^{nm}+t^{nm+\lambda}$, $\rho_n(y)=t^{(n+1)m}$, 
$\rho_n(z)=t^{(n+2)m}$, where $m=(n+1)/2$ and $\lambda$ is an integer greater than $n(n+1)m$, with $(\lambda, m)=1$. His proof relies in several technical intermediate results, one of them describing the kernel of a linear map given 
by a binomial matrix (see \cite[Theorem~4.2]{moh1}). Moh's paper is an invitation to deep in these matters as he leaves some open questions, such as to avoid the hypothesis on the field to be of characteristic zero, to avoid the restriction $n$ is odd, or to modify the exponents of the map $\rho_n$, so that one can find explicitly the generators of these primes ideals $P_n$ (see, e.g., in these directions \cite{gp1}, \cite{gp3} and \cite{mss}).

Pursuing this goal, the authors of the present manuscript rapidly run into the need 
to calculate the (left) nullspaces of some binomial matrices. 
Here the excellent work of Gessel and Viennot \cite{gv} has become a magnificent source of inspiration, 
though our purposes are distinct. In \cite{gv}, Gessel and Viennot, among several other results, prove the positivity of binomial determinants and give closed formulae when the rows or the columns of the binomial matrix are consecutive (see \cite[Corollary~2, Proposition~10 and page~308]{moh1}). 

The main result of the present paper expresses a binomial determinant as the product of a positive rational number with the sum of binomial determinants of size one less (Theorem~\ref{size-reduction}). As a corollary, we reprove the positivity of binomial determinants (Corollary~\ref{cor-positivity}), a result obtained by Gessel and Viennot in \cite[Corollary~2]{gv}, there using a bijection between the binomial determinant and the number of certain non-intersecting paths.

Subsequently, we give closed formulae for binomial determinants when the rows or the columns are consecutive (Theorems~\ref{Jinterval} and \ref{Iinterval}). While the first formula coincides with the formula given in \cite[page~308]{gv}, the second one is expressed in different terms (see \cite[Proposition~10]{gv}). As a consequence, we recover a theorem of Moh (\cite[Theorem~2.2]{moh1}) about the binomial determinant when both rows and columns are consecutive (see Corollary~\ref{corollary-moh}). 

Binomial determinants with consecutive rows and almost consecutive columns and, vice versa, almost consecutive rows and consecutive columns are studied in Sections~\ref{sectionAlmostJ} and \ref{sectionAlmostI}. The closed formulae obtained in these two sections are used in Section~\ref{sectionNulls} to give simple explicit generators to left nullspaces of some binomial matrices, those which are of interest in generalizing the work of Moh (see \cite{gp3}). We finish the paper by commenting the consequences of the ``interchanging formula'', rows with columns, given by Gessel and Viennot in \cite[Proposition~14]{gv}.  

\section{Size reduction and positivity}\label{positivity}

We first collect all the terminology in use. 

\begin{notation}\label{notation-binomial}
Let $\mbn=\{0,1,2,\ldots\}$ be the set of non-negative integers. 

If $k,l\in\mbn$, $k\leq l$, let $[k,l]=\{k,k+1,\ldots,l\}$ stand for the set of consecutive natural
numbers from $k$ to $l$. When $k>l$, we understand $[k,l]=\emptyset$. 

Let $d,e\in\mbn$, $d,e\geq 1$. From now on, $I=\{i_1,\ldots,i_d\}$ and $J=\{j_1,\ldots,j_e\}$
are two subsets of non-negative integers, where we always assume that $0\leq
i_1<\ldots<i_d$ and $0\leq j_1<\ldots<j_e$. 

If $p\in\mbn$, $p\leq i_1$, set $I-p=\{i_1-p,\ldots,i_d-p\}$. 
If $q\in\mbn$, $q\geq i_d$, set $q-I=\{q-i_d,\ldots,q-i_1\}$.

When $d=e$, we will say that $J\leq I$ whenever $j_1\leq i_1,\ldots,j_d\leq i_d$. Moreover,
$J<I$ if, and only if $J\leq I$ and $J\neq I$.

Let $B=(b_{i,j})_{i,j\geq 0}$ be the infinite matrix of
binomial coefficients, where $b_{i,j}=\binom{i}{j}=\frac{i!}{j!\cdot(i-j)!}$.
If $j>i$, then $b_{i,j}=0$. Note that rows and columns of $B$ are counted starting from zero.

Let $B^I_J$ denote the $d\times e$ submatrix of the
binomial matrix $B$ defined by the (intersection of the) rows $i_1,\ldots,i_d$ 
and the columns $j_1,\ldots,j_e$ of $B$.

If $d=e$, set 
\begin{eqnarray*}
b^{I}_{J}=\det(B^I_J)=
\left|\begin{array}{ccc}b_{i_1,j_1}&\ldots&b_{i_1,j_d}\\\vdots&&\vdots\\b_{i_d,j_1}&\ldots&b_{i_d,j_d}\end{array}\right|\mbox{ or, following \cite{gv}, }
b^I_J=\binom{I}{J}=\left(\begin{array}{ccc}i_1&\ldots&i_d\\j_1&\ldots&j_d\end{array}\right).
\end{eqnarray*}

If $d=e$, let 
\begin{eqnarray*}
\pi^I_J=\frac{b_{i_1,j_1}\cdots b_{i_d,j_1}}{b_{j_1,j_1}\cdots b_{j_d,j_1}}.
\end{eqnarray*}
This quotient of binomial coefficients is a well-defined rational number (see also Remark~\ref{pinatural}). 
It appears, not with this terminology, in 
\cite[Theorem~2.2]{moh1}, as the expression of a binomial determinant. 
We later recover this result in Corollary~\ref{corollary-moh}. Observe that $\pi^I_J=0$ if and only if $i_1<j_1$.
Moreover, if $j_1=0$ or $J=I$, then $\pi^I_J=1$.
\end{notation}

The first result reduces the evaluation of $b^I_J$ to the case $J<I$, with $j_1=0$. 

\begin{lemma}[\sc Reduction to $J<I$, $j_1=0$]\label{reductionJ<I}
Let $I=\{i_1,\ldots,i_d\}$ and $J=\{j_1,\ldots,j_d\}$. 
\begin{itemize}
\item[$(a)$] If $J\not\leq I$, then $b^I_J=0$.
\item[$(b)$] If $J=I$, then $b^I_J=1$.
\item[$(c)$] If $J\leq I$, then $b^I_J=\pi^I_J\cdot b^{I-j_1}_{J-j_1}$.
\end{itemize}
\end{lemma}
\begin{proof}
Suppose that $J\not\leq I$, so there exists $k$, $1\leq k\leq d$, with
$j_k>i_k$. Then $b_{i_k,j_k}=0$. Since 
$i_1<\ldots <i_{k-1}<i_k<j_k<j_{k+1}<\ldots <j_d$, 
it follows that $b_{i_r,j_s}=0$, for
all $1\leq r\leq k$ and for all $k\leq s\leq d$. Therefore, 
\[
B^I_J=\left(\begin{array}{c|c}B_1&B_2\\\hline 
B_3&B_4\end{array}\right),
\]
for some submatrices $B_1$, $B_2$, $B_3$ and $B_4$, 
where $B_1$ is $k\times (k-1)$, $B_2=0$ is $k\times (d-k+1)$,
$B_3$ is $(d-k)\times (k-1)$ and $B_4$ is $(d-k)\times (d-k+1)$. 
Observe that 
$\rank(B_1|B_2)=\rank(B_1|0)=\rank(B_1)\leq k-1$, 
which implies that $\rank(B^I_J)\leq d-1$, so $b^I_J=0$. 
This proves $(a)$. It is readily seen that, if $J=I$, then $B^I_J$
is a lower triangular matrix with entries equal to 1 in the diagonal. 
So, in this case, $b^I_J=1$, which shows $(b)$.

Suppose that $J=I$. By $(b)$, $b^I_J=1$. Moreover, as said before, $\pi^I_J=1$. Hence, $I-j_1=J-j_1$ and, by $(b)$ again, $b^{I-j_1}_{J-j_1}=1$. Therefore, $b^I_J=1=\pi^I_Jb^{I-j_1}_{J-j_1}$ and $(c)$ holds in this particular case.

Suppose that $J<I$. If $j_1=0$, then $\pi^I_J=1$, $I-j_1=I$, $J-j_1=J$ and $b^I_J=\pi^I_J\cdot b^{I-j_1}_{J-j_1}$. 
Suppose that $j_1\neq 0$. Using
$b_{p,q}=(p/q)b_{p-1,q-1}$, one can factor out $i_r$ from the $r$-th
row of $b^I_J$ and $1/j_s$ from the $s$-th column of $b^I_J$. 
Proceeding recursively, we get:
\begin{eqnarray*}
b^I_J&&=\frac{i_1\cdots i_d}{j_1\cdots j_d}b^{I-1}_{J-1}=
\frac{i_1\cdots i_d}{j_1\cdots j_d} \frac{(i_1-1)\cdots
  (i_d-1)}{(j_1-1)\cdots (j_d-1)}b^{I-2}_{J-2}\\&&= \frac{i_1\cdots
  i_d}{j_1\cdots j_d} \frac{(i_1-1)\cdots (i_d-1)}{(j_1-1)\cdots
  (j_d-1)}\cdots \frac{(i_1-(j_1-1))\cdots
  (i_d-(j_1-1))}{(j_1-(j_1-1))\cdots
  (j_d-(j_1-1))}b^{I-j_1}_{J-j_1}\\&&= \frac{i_1(i_1-1)\cdots
  (i_1-(j_1-1))}{j_1(j_1-1)\cdots (j_1-(j_1-1))} \cdots
\frac{i_d(i_d-1)\cdots (i_d-(j_1-1))}{j_d(j_d-1)\cdots (j_d-(j_1-1))}
b^{I-j_1}_{J-j_1}.
\end{eqnarray*}
Dividing by $j_1!$ on the numerator and the denominator of each
fraction:
\begin{eqnarray*}
\frac{i_r(i_r-1)\cdots (i_r-(j_1-1))}{j_r(j_r-1)\cdots (j_r-(j_1-1))}=
\frac{\frac{i_r(i_r-1)\cdots (i_r-(j_1-1))}{j_1!}}
     {\frac{j_r(j_r-1)\cdots (j_r-(j_1-1))}{j_1!}}=
     \frac{\binom{i_r}{j_1}}{\binom{j_r}{j_1}}=\frac{b_{i_r,j_1}}{b_{j_r,j_1}}.
\end{eqnarray*}
Therefore, $b^I_J=\frac{b_{i_1,j_1}b_{i_2,j_1}\cdots b_{i_d,j_1}}
{b_{j_1,j_1}b_{j_2,j_1}\cdots b_{j_d,j_1}}b^{I-j_1}_{J-j_1}=\pi^I_J\cdot b^{I-j_1}_{J-j_1}$, which
proves item $(c)$. 
\end{proof}

The key idea in the proof of $(c)$ is the one used in \cite[Lemma~8]{gv}. 
Next we state and prove the main result of the present note.

\begin{theorem}[\sc Size reduction]\label{size-reduction}
Let $I=\{i_1,\ldots,i_d\}$ and $J=\{j_1,\ldots,j_d\}$. Suppose that $J\leq I$. Then 
\begin{eqnarray*}
b^I_J=\pi^I_J\cdot\left(\sum_{(k_2,\ldots,k_d)\in [i_1-j_1,i_2-j_1-1]\times\ldots\times[i_{d-1}-j_1,i_{d}-j_1-1]} 
b^{\{k_2,\ldots,k_d\}}_{\{j_2-j_1-1,\ldots,j_d-j_1-1\}}\right).
\end{eqnarray*}
\end{theorem}
\begin{proof}
First, suppose that $j_1=0$. Since $b_{p,q}=b_{p-1,q-1}+b_{p-1,q}$, it follows that
\begin{eqnarray*}
b^{\{p\}}_J=(b^p_0,b^{p}_{j_2},\ldots,b^{p}_{j_d})=
(0,b^{p-1}_{j_2-1},\ldots,b^{p-1}_{j_d-1})+(b^{p-1}_0,b^{p-1}_{j_2},\ldots,b^{p-1}_{j_d}).
\end{eqnarray*}
For $2\leq s\leq d$,
\begin{eqnarray*}
\left|\begin{array}{ccccc}
*&*&&*\\
b^{i_{s-1}}_{0}&b^{i_{s-1}}_{j_2}&\ldots&b^{i_{s-1}}_{j_d}\\
b^{i_{s}}_{0}&b^{i_{s}}_{j_2}&\ldots&b^{i_{s}}_{j_d}\\
*&*&&*
\end{array}\right|=
\left|\begin{array}{ccccc}
*&*&&*\\
b^{i_{s-1}}_{0}&b^{i_{s-1}}_{j_2}&\ldots&b^{i_{s-1}}_{j_d}\\
0&b^{i_{s}-1}_{j_2-1}&\ldots&b^{i_{s}-1}_{j_d-1}\\
*&*&&*
\end{array}\right|+
\left|\begin{array}{ccccc}
*&*&&*\\
b^{i_{s-1}}_{0}&b^{i_{s-1}}_{j_2}&\ldots&b^{i_{s-1}}_{j_d}\\
b^{i_{s}-1}_{0}&b^{i_{s}-1}_{j_2}&\ldots&b^{i_{s}-1}_{j_d}\\
*&*&&*
\end{array}\right|.
\end{eqnarray*}
If $i_{s-1}=i_{s}-1$, the most right hand determinant is zero. If $i_{s-1}<i_{s}-1$, proceed
as before:
\begin{eqnarray*}
\left|\begin{array}{ccccc}
*&*&&*\\
b^{i_{s-1}}_{0}&b^{i_{s-1}}_{j_2}&\ldots&b^{i_{s-1}}_{j_d}\\
b^{i_{s}-1}_{0}&b^{i_{s}-1}_{j_2}&\ldots&b^{i_{s}-1}_{j_d}\\
*&*&&*
\end{array}\right|=
\left|\begin{array}{ccccc}
*&*&&*\\
b^{i_{s-1}}_{0}&b^{i_{s-1}}_{j_2}&\ldots&b^{i_{s-1}}_{j_d}\\
0&b^{i_{s}-2}_{j_2-1}&\ldots&b^{i_{s}-2}_{j_d-1}\\
*&*&&*
\end{array}\right|+
\left|\begin{array}{ccccc}
*&*&&*\\
b^{i_{s-1}}_{0}&b^{i_{s-1}}_{j_2}&\ldots&b^{i_{s-1}}_{j_d}\\
b^{i_{s}-2}_{0}&b^{i_{s}-2}_{j_2}&\ldots&b^{i_{s}-2}_{j_d}\\
*&*&&*
\end{array}\right|.
\end{eqnarray*}
Recursively, we get a decomposition in $i_s-i_{s-1}$ summands. Namely:
\begin{eqnarray*}
\left|\begin{array}{ccccc}
*&*&&*\\
b^{i_{s-1}}_{0}&b^{i_{s-1}}_{j_2}&\ldots&b^{i_{s-1}}_{j_d}\\
b^{i_{s}}_{0}&b^{i_{s}}_{j_2}&\ldots&b^{i_{s}}_{j_d}\\
*&*&&*
\end{array}\right|=
\sum_{k=i_{s-1}}^{i_s-1}
\left|\begin{array}{ccccc}
*&*&&*\\
b^{i_{s-1}}_{0}&b^{i_{s-1}}_{j_2}&\ldots&b^{i_{s-1}}_{j_d}\\
0&b^{k}_{j_2-1}&\ldots&b^{k}_{j_d-1}\\
*&*&&*
\end{array}\right|.
\end{eqnarray*}
Now, given the determinant $b^I_J$, apply this decomposition, first to the last row. Then, proceed in the same way, row by row, in ascending order. Finally, 
develop the determinant by the first column, so that:
\begin{multline*}
b^I_J=\sum_{k_d=i_{d-1}}^{i_d-1}
\left|\begin{array}{ccccc}
*&*&&*\\
b^{i_{d-1}}_{0}&b^{i_{d-1}}_{j_2}&\ldots&b^{i_{d-1}}_{j_d}\\
0&b^{k_d}_{j_2-1}&\ldots&b^{k_d}_{j_d-1}
\end{array}\right|=
\sum_{k_d=i_{d-1}}^{i_d-1}
\sum_{k_{d-1}=i_{d-2}}^{i_{d-1}-1}
\left|\begin{array}{ccccc}
*&*&&*\\
b^{i_{d-2}}_{0}&b^{i_{d-2}}_{j_2}&\ldots&b^{i_{d-2}}_{j_d}\\
0&b^{k_{d-1}}_{j_2-1}&\ldots&b^{k_{d-1}}_{j_d-1}\\
0&b^{k_d}_{j_2-1}&\ldots&b^{k_d}_{j_d-1}
\end{array}\right|=\\
\ldots=\sum_{k_d=i_{d-1}}^{i_d-1}
\sum_{k_{d-1}=i_{d-2}}^{i_{d-1}-1}\ldots\sum_{k_2=i_{1}}^{i_2-1}
\left|\begin{array}{ccccc}
b^{i_{1}}_{0}&b^{i_{1}}_{j_2}&\ldots&b^{i_{1}}_{j_d}\\
0&b^{k_{2}}_{j_2-1}&\ldots&b^{k_{2}}_{j_d-1}\\
\vdots&\vdots&&\vdots\\
0&b^{k_{d-1}}_{j_2-1}&\ldots&b^{k_{d-1}}_{j_d-1}\\
0&b^{k_d}_{j_2-1}&\ldots&b^{k_d}_{j_d-1}
\end{array}\right|=
\sum_{s=2}^d
\sum_{k_{s}=i_{s-1}}^{i_{s}-1}
\left|\begin{array}{cccc}
b^{k_{2}}_{j_2-1}&\ldots&b^{k_{2}}_{j_d-1}\\
\vdots&&\vdots\\
b^{k_d}_{j_2-1}&\ldots&b^{k_d}_{j_d-1}
\end{array}\right|.
\end{multline*}
Since $j_1=0$, then $\pi^I_J=1$ and 
$b^I_J=
\pi^I_J\left(\sum_{s=2}^d\sum_{k_{s}=i_{s-1}}^{i_{s}-1}b^{\{k_2,\ldots,k_s\}}_{\{j_2-1,\ldots,j_d-1\}}\right)$. 
If $j_1\neq 0$, by Lemma~\ref{reductionJ<I}, then $b^I_J=\pi^I_Jb^{I'}_{J'}$, where $I'=\{i_1',\ldots,i_d'\}$, with $i'_s=i_s-j_1$ and 
$J'=\{j_1',\ldots,j_d'\}$, with $j_s'=j_s-j_1$. Since $j_1'=0$, applying the first case, we get
\begin{eqnarray*}
b^I_J=\pi^I_Jb^{I'}_{J'}=\pi^I_J\pi^{I'}_{J'}\left(\sum_{s=2}^d\sum_{k_{s}=
i'_{s-1}}^{i'_{s}-1}b^{\{k_2,\ldots,k_s\}}_{\{j'_2-1,\ldots,j'_d-1\}}\right)=
\pi^I_J\left(\sum_{s=2}^d\sum_{k_{s}=
i_{s-1}-j_1}^{i_{s}-j_1-1}b^{\{k_2,\ldots,k_s\}}_{\{j_2-j_1-1,\ldots,j_d-j_1-1\}}\right).
\end{eqnarray*}
\end{proof}

\begin{example}
Following the notation of \cite{gv}, and using Theorem~\ref{size-reduction}, we obtain, for instance, 
\begin{eqnarray*}
\left(\begin{array}{cccc}3&5&7&8\\0&3&5&7
\end{array}\right)=
\left(\begin{array}{ccc}3&5&7\\2&4&6
\end{array}\right)+
\left(\begin{array}{ccc}3&6&7\\2&4&6
\end{array}\right)+
\left(\begin{array}{ccc}4&5&7\\2&4&6
\end{array}\right)+
\left(\begin{array}{ccc}4&6&7\\2&4&6
\end{array}\right),
\end{eqnarray*}
where $(k_2,k_3,k_4)\in[3,4]\times [5,6]\times \{7\}=\{(3,5,7),(3,6,7),(4,5,7),(4,6,7)\}$.
Some of the determinants appearing in the decomposition given by Theorem~\ref{size-reduction} might be zero. This is the case with the first two summands in the next example.  
\begin{eqnarray*}
\left(\begin{array}{cccc}1&5&7&8\\0&3&5&7
\end{array}\right)&=&
\left(\begin{array}{ccc}1&5&7\\2&4&6
\end{array}\right)+
\left(\begin{array}{ccc}1&6&7\\2&4&6
\end{array}\right)+
\left(\begin{array}{ccc}2&5&7\\2&4&6
\end{array}\right)+
\left(\begin{array}{ccc}2&6&7\\2&4&6
\end{array}\right)\\
&&\left(\begin{array}{ccc}3&5&7\\2&4&6
\end{array}\right)+
\left(\begin{array}{ccc}3&6&7\\2&4&6
\end{array}\right)+
\left(\begin{array}{ccc}4&5&7\\2&4&6
\end{array}\right)+
\left(\begin{array}{ccc}4&6&7\\2&4&6
\end{array}\right).
\end{eqnarray*}
\end{example}

As a consequence of Theorem~\ref{size-reduction}, we get an alternative proof to the positivity of the binomial determinants, when $J\leq I$, given by Gessel and Viennot in \cite[Corollary~2]{gv}.

\begin{corollary}[\sc Positivity]\label{cor-positivity} 
Let $I=\{i_1,\ldots,i_d\}$ and $J=\{j_1,\ldots,j_d\}$. Then $b^I_J\geq 0$. 
Moreover, $b^I_J>0$ if, and only if, $J\leq I$.
\end{corollary}
\begin{proof}
By Lemma~\ref{reductionJ<I}, if $J\not\leq I$, then $b^I_J=0$ and, if $J=I$, then $b^I_J=1$. 
Thus, it remains to prove that if $J<I$, then $b^I_J>0$.
We proceed by induction on $d\geq 1$. If $d=1$, then $I=\{i_1\}$, $J=\{j_1\}$. Since $J<I$, 
it follows that $j_1<i_1$ and $b^I_J=b^{i_1}_{j_1}>0$. 
Suppose that $d>1$ and that the induction hypothesis holds.
By Theorem~\ref{size-reduction}, $b^I_J$ is the product of $\pi^I_J$, which satifsfies $\pi^I_J>0$, with
the sum of the binomial determinants $b^{\{k_2,\ldots,k_d\}}_{\{j_2-j_1-1,\ldots,j_d-j_1-1\}}$, 
where 
\begin{eqnarray*}
(k_2,\ldots,k_d)\in [i_1-j_1,i_2-j_1-1]\times\ldots\times [i_{d-1}-j_1,i_{d}-j_1-1].
\end{eqnarray*}
By Lemma~\ref{reductionJ<I},
\begin{itemize}
\item if ${\{j_2-j_1-1,\ldots,j_d-j_1-1\}}\not\leq {\{k_2,\ldots,k_d\}}$, then 
$b^{\{k_2,\ldots,k_d\}}_{\{j_2-j_1-1,\ldots,j_d-j_1-1\}}=0$;
\item if ${\{j_2-j_1-1,\ldots,j_d-j_1-1\}}={\{k_2,\ldots,k_d\}}$, then 
$b^{\{k_2,\ldots,k_d\}}_{\{j_2-j_1-1,\ldots,j_d-j_1-1\}}=1$.
\end{itemize}
By the induction hypothesis,
\begin{itemize}
\item if ${\{j_2-j_1-1,\ldots,j_d-j_1-1\}}<{\{k_2,\ldots,k_d\}}$, then
$b^{\{k_2,\ldots,k_d\}}_{\{j_2-j_1-1,\ldots,j_d-j_1-1\}}>0$.
\end{itemize}
A possible choice of ${\{k_2,\ldots,k_d\}}$ is the subset $\{i_2-j_1-1,\ldots,i_d-j_1-1\}$. Since $J<I$, 
it follows that ${\{j_2-j_1-1,\ldots,j_d-j_1-1\}}<\{i_2-j_1-1,\ldots,i_d-j_1-1\}$ and
$b^{\{i_2-j_1-1,\ldots,i_d-j_1-1\}}_{\{j_2-j_1-1,\ldots,j_d-j_1-1\}}>0$. We conclude that all the terms are either 
zero or $1$, and that there is at least one summand which is strictly positive. Thus, $b^I_J>0$.
\end{proof}

As an immediate consequence, we deduce that some submatrices of the binomial matrix have always maximal rank. 

\begin{corollary}\label{max-rank}
Let $d\geq 1$, $I=[i,i+d-1]$, $H=[0,i+d-1]$ and $J=\{j_1,\ldots,j_d\}\subset H$. 
Then $J\leq I$ and $b^I_J>0$. In particular, 
every $d\times d$ submatrix of $B^I_H$ has rank $d$.
\end{corollary}
\begin{proof}
Since $J\subset H$, then $j_d\leq i+d-1$ and, since $I$ is an interval, necessarily, $J\leq I$. 
By Corollary~\ref{cor-positivity}, $b^I_J>0$. Since this holds for any $J=\{j_1,\ldots,j_d\}\subset H$, it follows that
every $d\times d$ submatrix of $B^I_H$ has rank $d$.
\end{proof}

\section{A closed formula for consecutive columns}

We start the section by recovering a closed formula for the binomial determinant when $J$ is an interval (see \cite[page~308]{gv}).

\begin{theorem}[\sc Consecutive columns]\label{Jinterval}
Let $I=\{i_1,\ldots,i_d\}$ and $J=[j,j+d-1]$. Then,
\begin{eqnarray*}
b^I_J=\pi^I_J\cdot\frac{\prod_{1\leq k<\ell\leq d}(i_{\ell}-i_k)}{\prod_{k=0}^{d-1}k!}=
\frac{b_{i_1,j}\cdots b_{i_d,j}}{b_{j,j}\cdots b_{j+d-1,j}}
\cdot\frac{\prod_{1\leq k<\ell\leq d}(i_{\ell}-i_k)}{\prod_{k=0}^{d-1}k!}.
\end{eqnarray*}
\end{theorem}
\begin{proof}
By Lemma~\ref{reductionJ<I}, $b^I_J=\pi^I_Jb^{I-j}_{[0,d-1]}$. Thus, we can suppose $J=[0,d-1]$. Note that, in such a case
\begin{eqnarray*}
b^I_J&=&\frac{1}{0!1!\cdots (d-1)!}\left|\begin{array}{ccccc}
1&i_1&i_1(i_1-1)&\ldots&i_1(i_1-1)\cdots(i_1-d+2)\\
\vdots&\vdots&\vdots&&\vdots\\
1&i_d&i_d(i_d-1)&\ldots&i_d(i_d-1)\cdots(i_d-d+2)
\end{array}\right|=\\
&=&\frac{1}{\prod_{k=0}^{d-1}k!}\left|\begin{array}{ccccc}
1&i_1&i_1^2-i_1&\ldots&i_1^{d-1}+\sum_{k=1}^{d-2}\lambda_ki_1^k\\
\vdots&\vdots&\vdots&&\vdots\\
1&i_d&i_d^2-i_d&\ldots&i_1^{d-1}+\sum_{k=1}^{d-2}\lambda_ki_1^k
\end{array}\right|=\frac{1}{\prod_{k=0}^{d-1}k!}\left|\begin{array}{ccccc}
1&i_1&i_1^2&\ldots&i_1^{d-1}\\
\vdots&\vdots&\vdots&&\vdots\\
1&i_d&i_d^2&\ldots&i_1^{d-1}\end{array}\right|.
\end{eqnarray*}
The last step is as follows: start by adding to the third colum, the second one. Then, add to the fourth column three times the new second one and subtract two times the second one. Proceeding in this way, recursively, we arrive to the final Vandermonde determinant.
\end{proof}

As a consequence of both Theorems~\ref{size-reduction} and \ref{Jinterval}, we obtain the following curious result. 

\begin{corollary}
Let $0\leq i_1<\ldots<i_d$ a finite sequence of non-negative integers. Then,
\begin{eqnarray*}
\sum_{(k_{1,2},\ldots,k_{1,d})}
\sum_{(k_{2,3},\ldots,k_{2,d})}\ldots
\sum_{(k_{d-2,d-1},k_{d-2,d})}\sum_{k_{d-1,d}}1=\frac{\prod_{1\leq k<\ell\leq d}(i_{\ell}-i_k)}{\prod_{k=0}^{d-1}k!},
\end{eqnarray*}
where 
\begin{eqnarray*}
&&(k_{1,2},\ldots,k_{1,d})\in [i_1,i_2-1]\times\ldots\times[i_{d-1},i_{d}-1],\\
&&(k_{2,3},\ldots,k_{2,d})\in [k_{1,2},k_{1,3}-1]\times\ldots\times[k_{1,m-1},k_{1,d}-1],\\
&&\ldots\\
&&(k_{d-2,d-1},k_{d-2,d})\in [k_{d-3,d-2},k_{d-3,d-1}-1]\times[k_{d-3,d-1},k_{d-3,d}-1],\\
&&k_{d-1,d}\in [k_{d-2,d-1},k_{d-2,d}-1].
\end{eqnarray*}
\end{corollary}
\begin{proof}
Applying repeatedly Theorem~\ref{size-reduction}, with $I=\{i_1,\ldots,i_d\}$ and $J=[0,d-1]$, and using that $b^{\{k_{d-1,d}\}}_{\{0\}}=1$,
one deduces that $b^I_J$ coincides with the left hand side part of the equality. By Theorem~\ref{Jinterval}, $b^I_J$ also coincides with the right hand side part of the equality. 
\end{proof}

\section{A closed formula for consecutive rows}

Let us introduce a useful notation. 

\begin{notation}\label{notation-Ik}
Let $(I,J)$ be the pair of subsets of $\mbn$ given by $I=[i,i+d-1]$ and $J=\{j_1,\ldots,j_d\}$. 
Suppose that $J\leq I$. Set $I^{(0)}=I$ and $J^{(0)}=J$. For $1\leq k\leq d-1$, let $(I^{(k)},J^{(k)})$ be the pair defined as follows:
\begin{eqnarray*}
I^{(k)}=[i+k-1,i+d-2]-j_k\phantom{+}\mbox{ and }\phantom{+}J^{(k)}=\{j_{k+1},\ldots,j_d\}-j_k-1.
\end{eqnarray*}
For instance, if $k=1$, then $I^{(1)}=[i-j_1,i-j_1+d-2]$ and $J^{(1)}=\{j_2-j_1-1,\ldots,j_d-j_1-1\}$; if $k=d-1$, then
$I^{(d-1)}=\{i+d-2-j_{d-1}\}$ and $J^{(d-1)}=\{j_d-j_{d-1}-1\}$.
\end{notation}

\begin{remark}\label{ikjk}
Let $(I,J)$ be the pair $I=[i,i+d-1]$ and $J=\{j_1,\ldots,j_d\}$. Suppose that $J\leq I$. 
The following hold. 
\begin{itemize}
\item[$(a)$] $I^{(k)}$ is an interval and $J^{(k)}\leq I^{(k)}$, for every 
$1\leq k\leq d-1$. 
\item[$(b)$] $I^{(k+1)}=(I^{(k)})^{(1)}$ and $J^{(k+1)}=(J^{(k)})^{(1)}$, 
for every $1\leq k\leq d-2$.
\item[$(c)$] If $J=[j,j+d-1]$ is an interval, then $J^{(k)}=[0,d-k-1]$, 
for every $1\leq k\leq d-1$.
\end{itemize}
\end{remark}
\begin{proof}
Note that $(a)$ and $(b)$ also hold for $k=0$, but not $(c)$. Suppose $1\leq k\leq d-1$.
By definition, $I^{(k)}$ is an interval. By hypothesis, $J\leq I$. So $j_d\leq i+d-1$ and
$j_d-j_k-1\leq i+d-2-j_k$. Since $I^{(k)}$ is an interval, it follows that $J^{(k)}\leq I^{(k)}$. This shows $(a)$. According to Notation~\ref{notation-Ik},
\begin{eqnarray*}
&&I^{(k)}=[i+k-1,i+d-2]-j_k=[i-j_k+k-1,i-j_k+d-2]\phantom{+}\mbox{ and }\phantom{+}\\
&&J^{(k)}=\{j_{k+1},\ldots,j_d\}-j_k-1=\{j_{k+1}-j_k-1,\ldots,j_d-j_k-1\}.
\end{eqnarray*}
Suppose that $1\leq k\leq d-2$. Then
\begin{eqnarray*}
(I^{(k)})^{(1)}&=&[i-j_k+k-1,i-j_k+d-3]-(j_{k+1}-j_k-1)=\\&&[i+k,i+d-2]-j_{k+1}=I^{(k+1)}
\mbox{ and}\\
(J^{(k)})^{(1)}&=&\{j_{k+2}-j_k-1,\ldots,j_d-j_k-1\}-(j_{k+1}-j_k-1)=\\&&
\{j_{k+2},\ldots,j_d\}-j_{k+1}-1=J^{(k+1)}.
\end{eqnarray*}
This shows $(b)$. Finally, if $J=[j,j+d-1]$ is an interval, then $j_k=j+k-1$, 
$j_{k+1}=j+k$, $j_d=j+d-1$ and 
$J^{(k)}=[j+k,j+d-1]-(j+k-1)-1=[0,d-k-1]$.
\end{proof}

As a consequence of Theorem~\ref{size-reduction}, when $I$ is an interval, we improve \cite[Lemma~9]{gv}.

\begin{lemma}\label{lemmaIinterval}
Let $I=[i,i+d-1]$ and $J=\{j_1,\ldots,j_d\}$. Suppose that $J\leq I$. 
\begin{itemize}
\item[$(a)$] Then $b^I_J=\pi^I_J\cdot b^{[i-j_1,i-j_1+d-2]}_{\{j_2-j_1-1,\ldots,j_d-j_1-1\}}=
\pi^I_J\cdot b^{I^{(1)}}_{J^{(1)}}$.
\item[$(b)$] If $j_1=0$, then $b^I_J=b^{[i,i+d-2]}_{\{j_2-1,\ldots,j_d-1\}}$.
\item[$(c)$] If $J=[0,d-1]$, then $b^I_{J}=1$ and $\pi^I_J=1$.
\end{itemize}
\end{lemma}
\begin{proof}
Take in Theorem~\ref{size-reduction}, $i_k=i+k-1$. Then
\begin{eqnarray*}
[i_1-j_1,i_2-j_1-1]\times\ldots\times[i_{d-1}-j_1,i_d-j_1-1]=\{i-j_1\}\times\ldots\times\{i-j_1+d-2\},
\end{eqnarray*}
which is equal to the interval $[i-j_1,i-j_1+d-2]$. This shows $(a)$. 
If $j_1=0$, then $\pi^I_J=1$ (see Notation~\ref{notation-binomial}) and $(b)$ follows from $(a)$. 
Suppose now that $J=[0,d-1]$, so $\{j_2-1,\ldots,j_d-1\}=[0,d-2]$.
Let us prove, by induction on $d\geq 1$, 
that $b^I_J=1$. For $d=1$, $b^I_J=\det(b_{i,0})=1$.
Supppose that $d>1$. By item $(b)$,
$b^I_J=b^{[i,i+d-2]}_{[0,d-2]}$, which, by the induction hypothesis, is equal to $1$. 
Moreover, if $J=[0,d-1]$, then $j_1=0$ and $\pi^I_J=1$ as well. 
\end{proof}

Next we give a closed formula for a binomial determinant with consecutive rows in terms of the products of $\pi$'s. 
This is an alternative expression to the one given in \cite[Proposition~10]{gv}. 

\begin{theorem}[\sc Consecutive rows]\label{Iinterval}
Let $I=[i,i+d-1]$ and $J=\{j_1,\ldots,j_d\}$. Suppose that $J\leq I$. Then, for every
$1\leq k\leq d-1$, 
\begin{eqnarray*}
b^I_J=\prod_{\ell=0}^{k-1}\pi^{I^{(\ell)}}_{J^{(\ell)}}b^{I^{(k)}}_{J^{(k)}}=
\pi^I_Jb^{I^{(1)}}_{J^{(1)}}=
\ldots=
\pi^I_J\pi^{I^{(1)}}_{J^{(1)}}\cdots \pi^{I^{(d-2)}}_{J^{(d-2)}}b^{I^{(d-1)}}_{J^{(d-1)}}=
\prod_{\ell=0}^{d-1}\pi^{I^{(\ell)}}_{J^{(\ell)}}.
\end{eqnarray*}
\end{theorem}
\begin{proof} 
Let us prove $b^I_J=\prod_{\ell=0}^{k-1}\pi^{I^{(\ell)}}_{J^{(\ell)}}b^{I^{(k)}}_{J^{(k)}}$
by induction on $1\leq k\leq d-1$. Suppose that $k=1$. By Lemma~\ref{lemmaIinterval}, 
$b^I_J=\pi^I_Jb^{I^{(1)}}_{J^{(1)}}=\pi^{I^{(0)}}_{J^{(0)}}b^{I^{(1)}}_{J^{(1)}}$. 
Suppose that $2\leq k\leq d-1$. By the induction hypothesis, 
$b^I_J=\prod_{\ell=0}^{k-2}\pi^{I^{(\ell)}}_{J^{(\ell)}}b^{I^{(k-1)}}_{J^{(k-1)}}$.
By Remark~\ref{ikjk},~$(a)$, $I^{(k-1)}$ is an interval and $J^{(k-1)}\leq I^{(k-1)}$. Thus,
by Lemma~\ref{lemmaIinterval} again, 
$b^{I^{(k-1)}}_{J^{(k-1)}}=\pi^{I^{(k-1)}}_{J^{(k-1)}}
b^{(I^{(k-1)})^{(1)}}_{(J^{(k-1)})^{(1)}}$. By Remark~\ref{ikjk},~$(b)$,
$(I^{(k-1)})^{(1)}=I^{(k)}$ and $(J^{(k-1)})^{(1)}=J^{(k)}$. Therefore, 
$b^{I^{(k-1)}}_{J^{(k-1)}}=\pi^{I^{(k-1)}}_{J^{(k-1)}}b^{I^{(k)})}_{J^{(k)}}$ and 
$b^I_J=\prod_{\ell=0}^{k-2}\pi^{I^{(\ell)}}_{J^{(\ell)}}b^{I^{(k-1)}}_{J^{(k-1)}}=
\prod_{\ell=0}^{k-1}\pi^{I^{(\ell)}}_{J^{(\ell)}}b^{I^{(k)}}_{J^{(k)}}$.

Finally, note that $I^{(d-1)}=\{i+d-2-j_{d-1}\}$ 
and $J^{(d-1)}=\{j_d-j_{d-1}-1\}$, so $b^{I^{(d-1)}}_{J^{(d-1)}}=b^{\{i+d-2-j_{d-1}\}}_{\{j_d-j_{d-1}-1\}}=\pi^{I^{(d-1)}}_{J^{(d-1)}}$.
\end{proof}

As a corollary we get the value of the binomial determinant of two intervals $I$ and $J$, a result stated
by Moh, without a demonstration, in \cite[Theorem~2.2]{moh1}.

\begin{corollary}\label{corollary-moh}
Let $I=[i,i+d-1]$ and $J=[j,j+d-1]$. Then $b^I_J=\pi^I_J$. 
In particular, if $j=0$, then $b^{[i,i+d-1]}_{[0,d-1]}=1$, and if $j=1$, then
$b^{[i,i+d-1]}_{[1,d]}=b_{i+d-1,d}$.
\end{corollary}
\begin{proof}
Suppose that $J\not\leq I$, in other words, $i<j$. Then $\pi^I_J=0$ (see Notation~\ref{notation-binomial})
and, by Lemma~\ref{reductionJ<I}, $b^I_J=0$. Suppose that $J\leq I$. By Theorem~\ref{Iinterval}, 
$b^I_J=\pi^I_J\pi^{I^{(1)}}_{J^{(1)}}\cdots \pi^{I^{(d-1)}}_{J^{(d-1)}}$. By Remark~\ref{ikjk}, for all $1\leq k\leq d-1$, 
$J^{(k)}=[0,d-k-1]$. Thus, $\pi^{I^{(1)}}_{J^{(1)}}=\ldots=\pi^{I^{(d-1)}}_{J^{(d-1)}}=1$ and  $b^I_J=\pi^I_J$. 
In particular, if $j=0$, by Notation~\ref{notation-binomial} (see also Lemma~\ref{lemmaIinterval}), then $\pi^{[i,i+d-1]}_{[0,d-1]}=1$, so 
$b^{[i,i+d-1]}_{[0,d-1]}=\pi^{[i,i+d-1]}_{[0,d-1]}=1$.
If $j=1$, then $b^{[i,i+d-1]}_{[1,d]}=
\pi^{[i,i+d-1]}_{[1,d]}=\frac{b_{i,1}\cdots b_{(i+d-1),1}}{b_{1,1}\cdots b_{d,1}}=
\frac{(i+d-1)\cdots i}{d!}=\frac{(i+d-1)!}{d!\cdot (i-1)!}=b_{i+d-1,d}$.
\end{proof}

\begin{corollary}\label{cor-int-j1}
Let $I=[i,i+d-1]$ and $J=[j,j+d-1]$. Then $b^I_J=b^{I\setminus I\cap J}_{J\setminus I\cap J}$.
\end{corollary}
\begin{proof}
By Corollary~\ref{corollary-moh}, $b^I_J=\pi^I_J$.
It is easy to check that $\pi^I_J=\pi^{I\setminus I\cap J}_{J\setminus I\cap J}$. 
Since $I\setminus I\cap J$ and $J\setminus I\cap J$ are two intervals, 
again by Corollary~\ref{corollary-moh}, it follows that
$\pi^{I\setminus I\cap J}_{J\setminus I\cap J}=b^{I\setminus I\cap J}_{J\setminus I\cap J}$.
\end{proof}

\begin{remark}
The formula $b^I_J=b^{I\setminus I\cap J}_{J\setminus I\cap J}$ is not true if $J$ or $I$ 
are not intervals. For instance, $b^{[2,3]}_{\{0,2\}}=2$, whereas $b^{\{3\}}_{\{0\}}=1$. Similarly, 
$b^{\{1,3\}}_{[0,1]}=2$, whereas $b^{\{3\}}_{\{0\}}=1$. 
\end{remark}

\begin{remark}\label{pinatural}
As a consequence of Corollary~\ref{corollary-moh}, $\pi^I_J$ is natural if $I$ and $J$ are both intervals. However, the converse is not true, as $\pi^I_J=1$ when $J=I$. 
Observe that, 
if $I$ or $J$ are not intervals, then $\pi^I_J$ might be not natural, for instance, $\pi^{\{1,3\}}_{[1,2]}=3/2$ and 
$\pi^{[4,5]}_{\{1,3\}}=20/3$. 
\end{remark}

\section{Consecutive rows and almost consecutive columns}\label{sectionAlmostJ}

We consider now consecutive rows and almost consecutive columns, that is, $I$ is an interval and $J$ is an interval missing one element. 
We start with the following particular case which loosely speaking says 
$b^{[\alpha,\beta]}_{[0,\gamma]\setminus\{\delta\}}=b_{\beta-\delta,\gamma-\delta}$.

\begin{lemma}\label{lemmaJpunct}
Let $d\geq 2$, $i\geq 1$, $I=[i,i+d-2]$ and $J=[0,d-1]$. If $1\leq r\leq d$, then 
\begin{eqnarray*}
b^{[i,i+d-2]}_{[0,d-1]\setminus\{r-1\}}=b_{i+d-2-(r-1),(d-1)-(r-1)}=b_{i+d-r-1,d-r}.
\end{eqnarray*}
\end{lemma}
\begin{proof}
Since $i\geq 1$ and $1\leq r\leq d-1$, then $J\setminus\{r-1\}=[0,d-1]\setminus\{r-1\}\leq [i,i+d-2]=I$. 
Suppose that $r=1$. By Corollary~\ref{corollary-moh}, 
\begin{eqnarray*}
b^{[i,i+d-2]}_{[0,d-1]\setminus\{r-1\}}=b^{[i,i+d-2]}_{[1,d-1]}=b_{i+d-2,d-1}=b_{i+d-r-1,d-r}.
\end{eqnarray*}
Suppose that $r=2$. By Lemma~\ref{lemmaIinterval}, $(b)$, and Corollary~\ref{corollary-moh}, 
\begin{eqnarray*}
b^{[i,i+d-2]}_{[0,d-1]\setminus\{r-1\}}=b^{[i,i+d-2]}_{\{0\}\sqcup [2,d-1]}=
b^{[i,i+d-3]}_{[1,d-2]}=b_{i+d-3,d-2}=b_{i+d-r-1,d-r}.
\end{eqnarray*}
Suppose that $r=d$. By Corollary~\ref{corollary-moh},
\begin{eqnarray*}
b^{[i,i+d-2]}_{[0,d-1]\setminus\{d-1\}}=b^{[i,i+d-2]}_{[0,d-2]}=1=b_{i+d-r-1,d-r}.
\end{eqnarray*}
Thus, the result is true for $r=1$, $r=2$ and $r=d$.
If $d=2$, since $1\leq r\leq d$, then $r=1$ or $r=d$ and the result holds true too.
Therefore, we can suppose that $d\geq 3$ and that $3\leq r\leq d-1$. 
In such a case, 
$J\setminus\{r-1\}=[0,d-1]\setminus\{r-1\}=[0,r-2]\sqcup [r,d-1]$. 
By Lemma~\ref{lemmaIinterval}, $(a)$, 
$b^{[i,i+d-2]}_{[0,r-2]\sqcup [r,d-1]}
=\pi^{[i,i+d-2]}_{[0,r-2]\sqcup [r,d-1]}\cdot b^{[i,i+d-3]}_{[0,r-3]\sqcup [r-1,d-2]}$.
By Notation~\ref{notation-binomial},
$\pi^{[i,i+d-2]}_{[0,r-2]\sqcup [r,d-1]}=1$. By the induction hypothesis,
$b^{[i,i+d-3]}_{[0,r-3]\sqcup [r-1,d-2]}=
b^{[i,i+d-3]}_{[0,d-2]\setminus\{r-2\}}=
b_{i+d-3-(r-2),d-2-(r-2)}=b_{i+d-r-1,d-r}$. 
\end{proof}

More in general, we obtain: 

\begin{theorem}[\sc Consecutive rows, almost consecutive columns]\label{Almostcolums}
Let $d\geq 2$, let $I=[i,i+d-2]$ and $J=[j,j+d-1]$, with $j\leq i-1$. 
If $2\leq r\leq d-1$, then 
\begin{eqnarray*}
b^I_{J\setminus\{j+r-1\}}=\pi^I_{J\setminus\{j+r-1\}}b_{i+d-j-r-1,d-r}.
\end{eqnarray*}
If $r=1$ or $r=d$, then 
\begin{eqnarray*}
b^I_{J\setminus\{j\}}=\pi^{I}_{J\setminus\{j\}}
\phantom{+}\mbox{ and }\phantom{+}
b^I_{J\setminus\{j+d-1\}}=\pi^{I}_{J\setminus\{j+d-1\}}.
\end{eqnarray*}
\end{theorem}
\begin{proof}
Suppose that $2\leq r\leq d-1$. Since $j\leq i-1$, it follows that $j+d-1\leq i+d-2$ and
\begin{eqnarray*}
J\setminus \{j+r-1\}=[j,j+r-2]\sqcup [j+r,j+d-1]\leq [i,i+d-2]=I.
\end{eqnarray*}
By Lemma~\ref{reductionJ<I}, $(c)$, 
$b^{I}_{J\setminus\{j+r-1\}}=\pi^{I}_{J\setminus\{j+r-1\}}b^{I-j}_{J\setminus\{j+r-1\}-j}$, where 
\begin{eqnarray*}
I-j=[i-j,i-j+d-2]\phantom{+}\mbox{ and }\phantom{+}J\setminus\{j+r-1\}-j=[0,d-1]\setminus\{r-1\}.
\end{eqnarray*}
Setting $\tilde{i}=i-j$, and by Lemma~\ref{lemmaJpunct},
$b^{I-j}_{J\setminus\{j+r-1\}-j}=b^{[\tilde{i},\tilde{i}+d-2]}_{[0,d-1]\setminus\{r-1\}}=
b_{\;\tilde{i}+d-r-1,d-r}=b_{i+d-j-r-1,d-r}$.

If $r=1$, then $J\setminus\{j+r-1\}=[j+1,j+d-1]$ and, by Corollary~\ref{corollary-moh}, 
$b^I_{J\setminus\{j\}}=\pi^{I}_{J\setminus\{j\}}$.

If $r=d$, then $J\setminus\{j+r-1\}=[j,j+d-2]$ and, by Corollary~\ref{corollary-moh}, 
$b^I_{J\setminus\{j+d-1\}}=\pi^{I}_{J\setminus\{j+d-1\}}$.
\end{proof}

\section{Almost consecutive rows and consecutive columns}\label{sectionAlmostI}

Now we study the case of almost consecutive rows and consecutive columns, i.e., 
$I$ is an interval missing one element and $J$ is an interval. Note that
\begin{eqnarray*}
[i,i+d-1]\setminus\{i+d-1\}<[i,i+d-1]\setminus\{i+d-2\}<\ldots<[i,i+d-1]\setminus\{i\}.
\end{eqnarray*}
In particular, if $J\leq [i,i+d-1]\setminus\{i+d-1\}$, then $J\leq [i,i+d-1]\setminus\{i+r-1\}$, 
for all $1\leq r\leq d$.

\begin{lemma}\label{lemma-binomialsum}
Let $d\geq 3$, $I=[i,i+d-1]$ and $J=\{0,j_2,\ldots,j_{d-1}\}$, 
with $J\leq I\setminus\{i+d-1\}$. If $2\leq r\leq d-1$, then
\begin{eqnarray*}
b^{I\setminus\{i+r-1\}}_J=
b^{[i,i+d-2]\setminus\{i+r-1\}}_{\{j_2,\ldots,j_{d-1}\}-1}+
b^{[i,i+d-2]\setminus\{i+r-2\}}_{\{j_2,\ldots,j_{d-1}\}-1}.
\end{eqnarray*}
If $r=1$ or $r=d$, then
\begin{eqnarray*}
b^{I\setminus\{i\}}_J=b^{[i+1,i+d-2]}_{\{j_2,\ldots,j_{d-1}\}-1}
\phantom{+}\mbox{ and }\phantom{+}
b^{I\setminus\{i+d-1\}}_J=b^{[i,i+d-3]}_{\{j_2,\ldots,j_{d-1}\}-1}.
\end{eqnarray*}
\end{lemma}
\begin{proof}
Suppose $2\leq r\leq d-1$. Write 
\begin{eqnarray*}
I\setminus\{i+r-1\}=[i,i+d-1]\setminus\{i+r-1\}=[i,i+r-2]\sqcup [i+r,i+d-1]=\{i_1,\ldots,i_{d-1}\},
\end{eqnarray*}
where $i_1=i,\ldots,i_{r-1}=i+r-2,i_r=i+r,\ldots,i_{d-1}=i+d-1$.
Then, following the notations in Theorem~\ref{size-reduction},
\begin{eqnarray*}
&&[i_1,i_2-1]\times\ldots\times[i_{r-2},i_{r-1}-1]\times [i_{r-1},i_r-1]\times 
[i_r,i_{r+1}-1]\times\ldots\times [i_{d-2},i_{d-1}-1]=\\
&&\{i\}\times\ldots\times\{i+r-3\}\times\{i+r-2,i+r-1\}\times\{i+r\}
\times\ldots\times\{i+d-2\}.
\end{eqnarray*}
Therefore, $(k_2,\ldots,k_{d-1})$ is either equal to $(i,\ldots,i+r-3,i+r-2,i+r,\ldots,i+d-2)$, or else to 
$(i,\ldots,i+r-3,i+r-1,i+r,\ldots,i+d-2)$. Hence, by Theorem~\ref{size-reduction}, 
\begin{eqnarray*}
b^{I\setminus\{i+r-1\}}_J&=&
b^{\{i,\ldots,i+r-3,i+r-2,i+r,\ldots,i+d-2\}}_{\{j_2,\ldots,j_{d-1}\}-1}+
b^{\{i,\ldots,i+r-3,i+r-1,i+r,\ldots,i+d-2\}}_{\{j_2,\ldots,j_{d-1}\}-1}=\\
&&b^{[i,i+d-2]\setminus\{i+r-1\}}_{\{j_2,\ldots,j_{d-1}\}-1}+
b^{[i,i+d-2]\setminus\{i+r-2\}}_{\{j_2,\ldots,j_{d-1}\}-1}.
\end{eqnarray*}
If $r=1$, then, by Lemma \ref{lemmaIinterval}, $(b)$, 
$b^{I\setminus\{i+r-1\}}_J=b^{I\setminus\{i\}}_J=b^{[i+1,i+d-1]}_{\{0,j_2,\ldots,j_{d-1}\}}=
b^{[i+1,i+d-2]}_{\{j_2,\ldots,j_{d-1}\}-1}$. Similarly, if $r=d$, then, 
by Lemma \ref{lemmaIinterval}, $(b)$, 
$b^{I\setminus\{i+r-1\}}_J=b^{I\setminus\{i+d-1\}}_J=b^{[i,i+d-2]}_{\{0,j_2,\ldots,j_{d-1}\}}=
b^{[i,i+d-3]}_{\{j_2,\ldots,j_{d-1}\}-1}$. 
\end{proof}

\begin{theorem}[\sc Almost consecutive rows, consecutive columns]\label{Almostrow}
Let $d\geq 2$, let $I=[i,i+d-1]$, $J=[j,j+d-2]$, with $j\leq i$. If $1\leq r\leq d$, 
then
\begin{eqnarray*}
b^{I\setminus\{i+r-1\}}_J=\pi^{I\setminus\{i+r-1\}}_Jb_{d-1,r-1}.
\end{eqnarray*}
\end{theorem}
\begin{proof} By Lemma~\ref{reductionJ<I}, $(c)$,
$b^{I\setminus\{i+r-1\}}_J=\pi^{I\setminus\{i+r-1\}}_Jb^{I\setminus\{i+r-1\}-j}_{J-j}$, 
where 
\begin{eqnarray*}
I\setminus\{i+r-1\}-j=[i-j,i-j+d-1]\setminus\{i-j+r-1\}=[\tilde{i},\tilde{i}+d-1]\setminus\{\tilde{i}+r-1\},
\end{eqnarray*}
$\tilde{i}=i-j$, and $J-j=[0,d-2]$. Thus, it is enough to prove, by induction on $d\geq 2$, that 
$b^{[i,i+d-1]\setminus\{i+r-1\}}_{[0,d-2]}=b_{d-1,r-1}$, for all $1\leq r\leq d$. 

If $d=2$, then $I=[i,i+1]$ and $J=\{0\}$. If $r=1$, 
$b^{I\setminus\{i+r-1\}}_J=b^{I\setminus\{i\}}_J
=b_{i+1,0}=1$ and $b_{d-1,r-1}=b_{1,0}=1$. If $r=2$, $b^{I\setminus\{i+r-1\}}_J=b^{I\setminus\{i+1\}}_J
=b_{i,0}=1$ and $b_{d-1,r-1}=b_{1,1}=1$. 

Suppose that $d\geq 3$ and that the induction hypothesis holds, so that
$b^{[i,i+s-1]\setminus\{i+r-1\}}_{[0,s-2]}=b_{s-1,r-1}$, for all $2\leq s\leq d-1$ and all 
$1\leq r\leq s$. 

Observe that, for any $1\leq r\leq d$, then 
$[0,d-2]\leq [i,i+d-1]\setminus\{i+r-1\}$.

If $r=1$, then $I\setminus\{i+r-1\}=I\setminus\{i\}=[i+1,i+d-1]$ and,  
by Lemma~\ref{lemmaIinterval}, $(c)$, $b^{[i+1,i+d-1]}_{[0,d-2]}=1=b_{d-1,0}=b_{d-1,r-1}$. 
Similarly, if $r=d$, then $I\setminus\{i+r-1\}=I\setminus\{i+d-1\}=[i,i+d-2]$ and,  
by the same result, $b^{[i,i+d-2]}_{[0,d-2]}=1=b_{d-1,d-1}=b_{d-1,r-1}$. 

Assume that $2\leq r\leq d-1$. By Lemma~\ref{lemma-binomialsum}, 
$b_{[0,d-2]}^{I\setminus\{i+r-1\}}=
b^{[i,i+d-2]\setminus\{i+r-1\}}_{[0,d-3]}+
b^{[i,i+d-2]\setminus\{i+r-2\}}_{[0,d-3]}$. 
By the induction hypothesis, 
$b^{[i,i+d-2]\setminus\{i+r-1\}}_{[0,d-3]}=b_{d-2,r-1}$
and $b^{[i,i+d-2]\setminus\{i+r-2\}}_{[0,d-3]}=b_{d-2,r-2}$.
Thus, $b_{[0,d-2]}^{I\setminus\{i+r-1\}}=b_{d-2,r-2}+b_{d-2,r-1}=b_{d-1,r-1}$.
\end{proof}

As a consequence, we obtain the following result which is 
widely used in \cite[Section~5]{gp3}.

\begin{corollary}
Let $d\geq 2$, $I=[i,i+d-1]$, $J=[0,d-2]$. If $1\leq r\leq d$, then 
\begin{eqnarray*}
b^{I\setminus\{i+r-1\}}_J=b_{d-1,r-1}.
\end{eqnarray*}
\end{corollary}
\begin{proof}
Since $j_1=0$, then $\pi^{I\setminus\{i+r-1\}}_J=1$ (see Notation~\ref{notation-binomial}).
The result follows from Theorem~\ref{Almostrow}.
\end{proof}

\section{Rank and left nullspaces of some binomials matrices}\label{sectionNulls}

This section is devoted to calculate the left nullspaces of some binomial matrices. 
This has strong connections with the work in \cite{gp3} 
(see, e.g., \cite[Lemma~3.5, $(3)$]{gp3}). Next we prove the main result of the section. 

\begin{theorem}\label{theorem-bingoods}
Let $d\geq 3$, $I=[i,i+d-1]$ and 
$J=\{0\}\sqcup [j,j+d-3]$, with 
$J\leq I\setminus\{i+d-1\}$. 
Set $\lambda=(\prod_{k=1}^{d-1}b_{i+k-1,j-1})/
(\prod_{k=0}^{d-3}b_{j+k-1,j-1})$. 
\begin{itemize}
\item[$(a)$] If $2\leq r\leq d-1$, then 
\begin{eqnarray*}
b^{I\setminus\{i+r-1\}}_J=\lambda\left(
\frac{b_{d-2,r-1}}{b_{i+r-1,j-1}}+
\frac{b_{d-2,r-2}}{b_{i+r-2,j-1}}\right).
\end{eqnarray*}
\item[$(b)$] If $r=1$ or $r=d$, then 
\begin{eqnarray*}
b^{I\setminus\{i\}}_J=\frac{\lambda}{b_{i,j-1}}
\phantom{+}\mbox{ and }\phantom{+}
b^{I\setminus\{i+d-1\}}_J=\frac{\lambda}{b_{i+d-2,j-1}}.
\end{eqnarray*}
\end{itemize}
\end{theorem}
\begin{proof}
Suppose that $2\leq r\leq d-1$. 
Using Lemma~\ref{lemma-binomialsum} and Theorem~\ref{Almostrow}, we get:
\begin{eqnarray*}
b^{I\setminus\{i+r-1\}}_J&=&
b^{[i,i+d-2]\setminus\{i+r-1\}}_{[j-1,j+d-4]}+
b^{[i,i+d-2]\setminus\{i+r-2\}}_{[j-1,j+d-4]}=\\
&&\pi^{[i,i+d-2]\setminus\{i+r-1\}}_{[j-1,j+d-4]}\cdot b_{d-2,r-1}+
\pi^{[i,i+d-2]\setminus\{i+r-2\}}_{[j-1,j+d-4]}\cdot b_{d-2,r-2}=\\ 
&&\frac{\prod_{k=1,k\neq r}^{d-1}b_{i+k-1,j-1}}
{\prod_{k=0}^{d-3}b_{j+k-1,j-1}}\cdot b_{d-2,r-1}+
\frac{\prod_{k=1,k\neq r-1}^{d-1}b_{i+k-1,j-1}}
{\prod_{k=0}^{d-3}b_{j+k-1,j-1}}\cdot b_{d-2,r-2}=\\&&
\frac{\prod_{k=1}^{d-1}b_{i+k-1,j-1}}{\prod_{k=0}^{d-3}b_{j+k-1,j-1}}
\left(\frac{b_{d-2,r-1}}{b_{i+r-1,j-1}}+\frac{b_{d-2,r-2}}{b_{i+r-2,j-1}}\right)=
\lambda\left(\frac{b_{d-2,r-1}}{b_{i+r-1,j-1}}+\frac{b_{d-2,r-2}}{b_{i+r-2,j-1}}\right).
\end{eqnarray*}
This shows $(a)$. Suppose that $r=1$. By Lemma~\ref{lemmaIinterval}, $(b)$, and 
Corollary~\ref{corollary-moh}, then 
\begin{eqnarray*}
b^{[i,i+d-1]\setminus\{i\}}_{\{0\}\cup[j,j+d-3]}=
b^{[i+1,i+d-1]}_{\{0\}\cup[j,j+d-3]}=
b^{[i+1,i+d-2]}_{[j-1,j+d-4]}=
\frac{\prod_{k=2}^{d-1}b_{i+k-1,j-1}}{\prod_{k=0}^{d-3}b_{j+k-1,j-1}}=\frac{\lambda}{b_{i,j-1}}.
\end{eqnarray*}
Similarly, if $r=d$, by Lemma~\ref{lemmaIinterval}, $(b)$, and Corollary~\ref{corollary-moh},
then 
\begin{eqnarray*}
b^{[i,i+d-1]\setminus\{i+d-1\}}_{\{0\}\cup[j,j+d-3]}=
b^{[i,i+d-2]}_{\{0\}\cup[j,j+d-3]}=
b^{[i,i+d-3]}_{[j-1,j+d-4]}=
\frac{\prod_{k=1}^{d-2}b_{i+k-1,j-1}}{\prod_{k=0}^{d-3}b_{j+k-1,j-1}}=\frac{\lambda}{b_{i+d-2,j-1}}.
\end{eqnarray*}
\vspace*{-0.3cm}
\end{proof}


Let $\mbk$ be a field of characteristic zero and let $u_1,\ldots,u_d$ denote the natural basis of $\mbk^d$. For any vector $x=x_1u_1+\ldots+x_du_d\in\mbk^d$, $x_i\in\mbk$, let $X$ denote the 
column vector of its components, so $X^{\top}=(x_1,\ldots,x_d)$. 
Let us calculate a basis for the left nullspace of some binomial matrices $B^I_J$, where $\card(I)=d$.
We understand by the left nullspace of an $d\times e$ matrix $B$ to the kernel 
$\ker(B^\top)=\{x\in\mbk^d\mid X^{\top}B=0\}$.

\begin{corollary}\label{cor-ker}
Let $d\geq 3$, $I=[i,i+d-1]$ and $J=\{0\}\sqcup [j,j+d-3]$,
with $J\leq I\setminus\{i+d-1\}$. 
Then the left nullspace of $B^I_J$ is generated by the vector
\begin{eqnarray*}
\frac{1}{b_{i,j-1}}u_1+
\sum_{r=2}^{d-1}(-1)^{r-1}
\left(
\frac{b_{d-2,r-1}}{b_{i+r-1,j-1}}+
\frac{b_{d-2,r-2}}{b_{i+r-2,j-1}}\right)u_r+
(-1)^{d-1}\frac{1}{b_{i+d-2,j-1}}u_d.
\end{eqnarray*}
\end{corollary}
\begin{proof}
Since $\card(I)=d$ and $\card(J)=d-1$, $B^I_J$ is a $d\times (d-1)$ matrix.
Applying Corollary~\ref{cor-positivity} to $J\leq I\setminus\{i+d-1\}=[i,i+d-2]$,
we deduce that $b^{[i,i+d-2]}_J>0$.
Thus, $\rank(B^{[i,i+d-1]}_J)=d-1$ and $\dim\ker((B^I_J)^{\top})=1$. 
By Cramer's rule, $\ker((B^I_J)^{\top})=\se\sum_{r=1}^{d}(-1)^{r-1}b^{I\setminus\{i+r-1\}}_Ju_r\sd$. 
Then the result follows from Theorem~\ref{theorem-bingoods}. 
\end{proof}

We illustrate the result above with two examples. 

\begin{example}
Let $d\geq 2$, $I=[0,d-1]$ and $J=[0,d-2]$. 
Then the left nullspace of $B^I_J$ is generated by the vector 
\begin{eqnarray*}
\sum_{r=1}^{d}(-1)^{r-1}b_{d-1,r-1}u_r.
\end{eqnarray*}
\end{example}
\begin{proof}
If $d\geq 3$, take $i=0$ and $j=1$ in Corollary~\ref{cor-ker}. Then,
$b_{i,j-1}=b_{0,0}=1$, $b_{i+r-1,j-1}=b_{r-1,0}=1$ and $b_{i+r-2,j-1}=b_{r-2,0}=1$,
for all $2\leq r\leq d-1$, and $b_{i+d-2,j-1}=b_{d-2,0}=1$.
Moreover, $b_{d-2,r-1}+b_{d-2,r-2}=b_{d-1,r-1}$, for all $2\leq r\leq d-1$. If $d=2$, then 
$\sum_{r=1}^{2}(-1)^{r-1}b_{d-1,r-1}u_r=u_1-u_2$, which is clearly a basis of the left nullspace of $B^I_J$. 
\end{proof}

\begin{example}
Let $d\geq 3$, $I=[1,d]$ and $J=\{0\}\sqcup [2,d-1]$.
Then the left nullspace of $B^I_J$ is generated by the vector
\begin{eqnarray*}
b_{d-1,1}u_1+\sum_{r=2}^{d}(-1)^{r-1}b_{d,r}u_r.
\end{eqnarray*}
\end{example}
\begin{proof}
Take $i=1$ and $j=2$ in Corollary~\ref{cor-ker}. Then 
$b_{i,j-1}=b_{1,1}=1$, $b_{i+r-1,j-1}=b_{r,1}=r$, $b_{i+r-2,j-1}=b_{r-1,1}=r-1$ and
$b_{i+d-2,j-1}=b_{d-1,1}=d-1$. Therefore,
the left nullspace of $B^I_J$ is generated by the vector
\begin{eqnarray*}
u_1+\sum_{r=2}^{d-1}(-1)^{r-1}\left(\frac{b_{d-2,r-1}}{r}+\frac{b_{d-2,r-2}}{r-1}\right)u_r+(-1)^{d-1}\frac{1}{d-1}u_d.
\end{eqnarray*}
Multiplying this vector by $d-1$ and using that $b_{p,q}=(p/q)b_{p-1,q-1}$ and 
$b_{d-1,r}+b_{d-1,r-1}=b_{d,r}$, one gets:
\begin{eqnarray*}
&&(d-1)u_1+\sum_{r=2}^{d-1}(-1)^{r-1}
\left(\frac{d-1}{r}b_{d-2,r-1}+\frac{d-1}{r-1}b_{d-2,r-2}\right)u_r+(-1)^{d-1}u_d=\\
&&b_{d-1,1}u_1+\sum_{r=2}^{d-1}(-1)^{r-1}b_{d,r}u_r+(-1)^{d-1}b_{d,d}u_d=
b_{d-1,1}u_1+\sum_{r=2}^{d}(-1)^{r-1}b_{d,r}u_r.
\end{eqnarray*}
\end{proof}

\begin{remark}
In Theorem~\ref{theorem-bingoods} one suposes that $J\leq I\setminus\{i+d-1\}$
in order to ensure that the binomial coefficients appearing in the proof are non-zero. 
However, similar results can be obtained removing such hypothesis. 
\end{remark}

\begin{example}
Let $d\geq 2$, $I=[0,d-1]$ and $J=\{0\}\sqcup [2,d-1]$. 
Then the left nullspace of $B^I_J$ is generated by the vector $u_1-u_2$.
\end{example}
\begin{proof}
Clearly, $B^I_J$ is a $d\times (d-1)$ matrix and $J\leq [1,d-1]=I\setminus\{0\}$. 
By Corollary~\ref{cor-positivity}, $b^{I\setminus\{0\}}_J>0$. 
Thus, $\rank(B^{I}_J)=d-1$ and $\dim\ker((B^I_J)^{\top})=1$. 
By Cramer's rule, $\ker((B^I_J)^{\top})=\se\sum_{r=1}^{d}(-1)^{r-1}b^{I\setminus\{r-1\}}_Ju_r\sd$. 
If $r=1$, by Lemmas~\ref{lemmaIinterval} and \ref{reductionJ<I}, 
$b^{I\setminus\{r-1\}}_J=b^{[1,d-1]}_{\{0\}\sqcup [2,d-1]}=b^{[1,d-2]}_{[1,d-2]}=1$.
If $r=2$, by Lemma~\ref{reductionJ<I}, 
$b^{I\setminus\{r-1\}}_J=b^{\{0\}\sqcup [2,d-1]}_{\{0\}\sqcup [2,d-1]}=1$. If $r\geq 3$, then 
$J=\{0,2,\ldots,d-1\}\not\leq [0,d-1]\setminus\{r-1\}$. 
By Lemma~\ref{reductionJ<I} again, $b^{I\setminus\{r-1\}}_J=0$. 
\end{proof}

\section{Interchanging rows and columns}

We finish the present paper with a result of Gessel and Viennot about interchanging rows and columns and 
the effect it may have on some former results. 
Given $I=\{i_1,\ldots,i_d\}$, $J=\{j_1,\ldots,j_d\}$ and $n\geq j_d,i_d$, let
\begin{eqnarray*}
q^J_I(n)=\frac{b_{n,j_1}\cdots b_{n,j_d}}{b_{n,i_1}\cdots b_{n,i_d}}.
\end{eqnarray*}
Note that $q^J_I(n)q^I_J(n)=1$. Next result, shown in \cite[Proposition~14]{gv}, 
says how to interchange rows and columns in a binomial determinant. We prove it here, 
with our notations, for the sake of completeness. 

\begin{proposition}
Let $I=\{i_1,\ldots,i_d\}$, $J=\{j_1,\ldots,j_d\}$, 
with $J\leq I$, and let $n\geq i_d$. Then 
\begin{eqnarray*}
b^I_J=q^J_I(n)b^{n-J}_{n-I}.
\end{eqnarray*}
\end{proposition}
\begin{proof}
Let $\Theta_d$ be the $d\times d$ exchange matrix, that is, 
the $d\times d$ antidiagonal matrix with 1's on the diagonal going from the lower left corner 
to the upper right corner. 
Since $\Theta^2=1$, it follows that $b^I_J=\det(\Theta_d(B^I_J)^{\top}\Theta_d)$, where
\begin{eqnarray*}
\Theta_d(B^I_J)^{\top}\Theta_d=
\left(\begin{array}{ccc}0&\ldots&1\\\vdots&&\vdots\\1&\ldots&0\end{array}\right)\!
\left(\begin{array}{ccc}
b_{i_1,j_1}&\ldots&b_{i_d,j_1}\\\vdots&&\vdots\\b_{i_1,j_d}&\ldots&b_{i_d,j_d}\end{array}\right)\!
\left(\begin{array}{ccc}0&\ldots&1\\\vdots&&\vdots\\1&\ldots&0\end{array}\right)=
\left(\begin{array}{ccc}
b_{i_d,j_d}&\ldots&b_{i_1,j_d}\\\vdots&&\vdots\\
b_{i_d,j_1}&\ldots&b_{i_d,j_1}
\end{array}\right).
\end{eqnarray*}
It is easy to check that, for $q\leq p\leq n$, then
$b_{n,p}b_{p,q}=b_{n,q}b_{n-q,n-p}$, so 
$b_{p,q}=b_{n,q}b_{n-q,n-p}/b_{n,p}$. It follows that
\begin{eqnarray*}
b^I_J=\left|\begin{array}{ccc}
\frac{b_{n,j_d}}{b_{n,i_d}}b_{n-j_d,n-i_d}&\ldots&
\frac{b_{n,j_d}}{b_{n,i_1}}b_{n-j_d,n-i_1}
\\\vdots&&\vdots\\
\frac{b_{n,j_1}}{b_{n,i_d}}b_{n-j_1,n-i_d}&\ldots&
\frac{b_{n,j_1}}{b_{n,i_1}}b_{n-j_1,n-i_1}
\end{array}\right|=\frac{b_{n,j_1}\cdots b_{n,j_d}}
{b_{n,i_1}\cdots b_{n,i_d}}b^{n-J}_{n-I}=q^{J}_{I}(n)\cdot b^{n-J}_{n-I}.
\end{eqnarray*}
\end{proof}

As an immediate consequence, we get:

\begin{corollary}
Let $I=\{i_1,\ldots,i_d\}$, $J=\{j_1,\ldots,j_d\}$, 
with $J\leq I$, and let $m\geq n\geq i_d$. Then 
\begin{eqnarray*}
b^I_J=q^J_I(n)q^{n-I}_{n-J}(m)b^{I+m-n}_{J+m-n}.
\end{eqnarray*}
\end{corollary}

As a consequence of Theorems~\ref{Jinterval} and \ref{Iinterval} and the interchanging procedure we obtain the following formula. 

\begin{corollary}
Let $I=[i,i+d-1]$ and $J=\{j_1,\ldots,j_d\}$. Let $n\geq i+d-1$. Then,
\begin{eqnarray*}
\pi^I_J\cdot\pi^{I^{(1)}}_{J^{(1)}}\cdots\pi^{I^{(d-1)}}_{J^{(d-1)}}=
q^J_I(n)\pi^{n-J}_{n-I}\frac{\prod_{1\leq k\leq \ell<d}(j_{\ell}-j_k)}{\prod_{k=0}^{d-1}k!}.
\end{eqnarray*}
\end{corollary}




{\small
\begin{thebibliography}{cc}
\bibitem{gv}{I. Gessel, G. Viennot, Binomial determinants, paths, and hook 
length formulae. Adv. in Math. {\bf 58} (1985), no.~3, 300-321.}
\bibitem{gp1}{L. Gonz\'alez, F. Planas-Vilanova, Prime ideals of Moh and the characteristic of the field. J. Algebra Appl. 
{https://www.worldscientific.com/doi/10.1142/S0219498827501702}}
\bibitem{gp3}{L. Gonz\'alez, F. Planas-Vilanova, Explicit minimal generating sets of a family of prime
ideals with unbounded minimal number of generators in a three-dimensional power series ring. In preparation.}
\bibitem{mss}{R. Mehta, J. Saha, I. Sengupta, 
 Moh's example of algebroid space curves. J. Symbolic Comput. {\bf 104} (2021), 
 168-182.}
\bibitem{moh1}{T.T. Moh, On the unboundedness of generators of prime
  ideals in power series rings of three variables. J. Math. Soc. Japan
  {\bf 26} (1974), 722-734}
\end{thebibliography}
}
\end{document}